\documentclass{article}
\usepackage[letterpaper, margin=1in]{geometry}
\usepackage{graphicx} %
\usepackage{color,amssymb,mathtools}
\usepackage{amsthm}
\usepackage{url}
\usepackage[inline]{enumitem}
\usepackage{cool}
\usepackage{physics}
\usepackage{tikz}
\usepackage{caption}
\usepackage{subcaption}
\usepackage{dsfont}
\usepackage{listings}
\usepackage{siunitx}
\usepackage{csquotes}
\usepackage{tabularx} 
\usepackage{booktabs}
\usetikzlibrary{automata, positioning, fit}
\usepackage{placeins} 
\usepackage[ntheorem]{empheq}
\usepackage{hyperref}
\usepackage{cleveref}
\usepackage[style=numeric,backend=biber]{biblatex}

\AtBeginDocument{\RenewCommandCopy\qty\SI}
\ExplSyntaxOn
\msg_redirect_name:nnn { siunitx } { physics-pkg } { none }
\ExplSyntaxOff

\DeclareSIUnit{\week}{weeks}

\ifpdf
\hypersetup{
  pdftitle={Micro-Macro Coupling for Optimizing Scaffold Mediated Bone Regeneration},
  pdfauthor={Patrick Dondl and Oliver Suchan}
}
\fi

\def\R{\mathbb{R}}
\def\C{\mathbb{C}}

\def\cB{\mathcal{B}}
\def\cC{\mathcal{C}}

\def\cR{\mathcal{R}}
\def\cS{\mathcal{S}}

\def\cM{\mathcal{M}}

\DeclareMathOperator{\diver}{div}

\def\sym{\nabla^{\mathrm{s}}}
\def\stiffness{\C}
\def\stiffnesshom{\stiffness^{\mathrm{hom}}}
\def\stiffnessmicro{\stiffness^{\mathrm{micro}}}

\def\strain{\varepsilon}

\NewDocumentCommand \qforall {s}
{
\IfBooleanTF {#1}
{\ensuremath{\quad \text{for all} \, }}
{\qq{for all}}
}

\DeclarePairedDelimiterX\setc[2]{\{}{\}}{\,#1 \;\delimsize\vert\; #2\,}

\newcommand{\cost}{c_{\textrm{ost}}}
\newcommand{\cpro}{c_{\textrm{pro}}}

\newcommand{\kdiffpro}{k_{\textrm{pro}}^{\textrm{diff}}}

\newcommand{\mufi}{\mu_{\textrm{fi}}}
\newcommand{\muost}{\mu_{\textrm{ost}}}
\newcommand{\mucho}{\mu_{\textrm{cho}}}
\newcommand{\oct}{\gamma_{\textrm{oct}}}

\newcommand{\weakc}{\rightharpoonup}

\newcommand{\GD}{\mathrm{D}}

\newcommand{\scaffold}{\rho}
\newcommand{\bone}{\mathrm{b}}
\newcommand{\uhom}{u^\mathrm{hom}}
\newcommand{\macstrain}{\sym \uhom}
\newcommand{\micstrain}{\sym u_\varepsilon}
\newcommand{\diff}{D}
\newcommand{\diffhom}{\diff^\mathrm{hom}}
\newcommand{\diffmic}{\diff^\mathrm{micro}}
\newcommand{\migspeed}{k_\mathrm{mig}}

\makeatletter

\Style{IntegrateDifferentialDSymb=\mathrm{d}}

\title{Micro-Macro Coupling for Optimizing Scaffold Mediated Bone Regeneration}
\author{Patrick Dondl\thanks{Department of Applied Mathematics, University of Freiburg, Freiburg, Germany.}
\and Oliver Suchan\footnotemark[1]
  }
\date{}

\theoremstyle{plain}
\newtheorem{theorem}{Theorem}[section]

\newtheorem{lemma}[theorem]{Lemma}

\theoremstyle{remark}
\newtheorem{remark}[theorem]{Remark}

\begin{document}

\maketitle

\begin{abstract} 
This work presents a framework for modeling three-dimensional scaffold-mediated bone regeneration and the associated optimization problem. By incorporating microstructure into the model through periodic homogenization, we capture the effects of microscale fluctuations on the bone growth process. Numerical results and optimized scaffold designs that explicitly account for the microstructure are presented, demonstrating the potential of this approach for improving scaffold performance. 
\end{abstract}

\section{Introduction}

In this work, we consider mathematical models for the regeneration of critical size bone aidded by implantation of bioresorbable porous scaffolds at the defect site. The model presented here is based on the experimentally validated framework proposed in \cite{MahdiJaber}. However, a key limitation of \cite{MahdiJaber} is its inability to accommodate a mathematical optimization framework. Such optimization was previously considered in \cite{ScaffoldOpt}, albeit with ad-hoc homogenized quantities, thus not explicitly incorporating the microstructural properties of bone and scaffold architecture.

To address these limitations, we employ a periodic homogenization approach that integrates microscale effects into the model. Specifically, we use a two-scale finite element method (FE$^2$) to model the reaction-diffusion processes and a finite element-fast Fourier transform (FE-FFT) method to evaluate the mechanical stimulation of cells via linear elasticity. Our primary focus is on the computational aspects of this micro-macro coupling approach. The rigorous mathematical analysis remains an area for future exploration.

The proposed framework is versatile and can accommodate various microstructures. In this work, we focus on Gyroid-based and strut-like architectures, which are particularly well-suited for scaffolding applications in bone regeneration. Cell responses to stimuli are modeled using coupled anisotropic, heterogeneous, and nonlinear reaction-diffusion equations, supplemented by logistic growth ordinary differential equations (ODEs). The mechanical stimulation of cells is driven by linear elasticity, with a fixator incorporated into the computational domain to ensure a stable mechanical environment conducive to bone regeneration, as described in \cite{MahdiJaber}. While the inclusion of the microstructure significantly increases computational complexity and runtime, the framework demonstrates promising results within manageable timeframes. Furthermore, its flexibility allows the consideration of a number of patient-specific parameters, paving the way for a precision medicine approach to  tissue engineering scaffold designs.

The remainder of this article is structured as follows: In the next section, we revisit the ad hoc homogenized mac\-roscopic model, introduce the problem at the microscopic scale, and derive the homogenized quantities as functions of microscale fluctuations. \Cref{sec:micro_macro_coupling} explores the coupling between the microscopic and macroscopic scales, completing the mathematical modeling framework. \Cref{sec:implementation_details} discusses the implementation of the model using the general-purpose finite element method (FEM) framework \lstinline{firedrake}. The integration of the microscopic model into the optimization process, including the derivation of adjoint-based derivatives for the homogenized quantities, is detailed in \Cref{sec:optimization}. Finally, \Cref{sec:simulations} presents and analyzes the results obtained with the proposed framework.

\section{Micro-Macro Model}

We start by revisiting the ad-hoc homogenized mac\-roscopic model for the scaffold mediated bone regeneration process which is based on \cite{ScaffoldOpt,ScaffBone} and adapted slightly. This ad-hoc model turns out to be insufficient for the description of bone regeneration and leads to the necessity to also consider the system on the microscopic scale. Resolving the microstructure explicitly, however, is computationally not feasible. Therefore, we approximate the heterogeneous microstructure by assuming a periodic pattern which reduces the problem significantly as we shall see in the subsequent sections. This is called \emph{Micro-Macro-coupling} because one needs to solve a (periodic) microscopic problem in each macroscopic space which will then be used to determine macroscopic properties, such as the elastic moduli, mechanical stimulus and diffusivity. \Cref{fig:micro_macro_coupling} visualizes this approach for a spatially varying bone volume fraction and a constant scaffold density.

\begin{figure}[ht]
    \centering
    \includegraphics[width=0.9\linewidth]{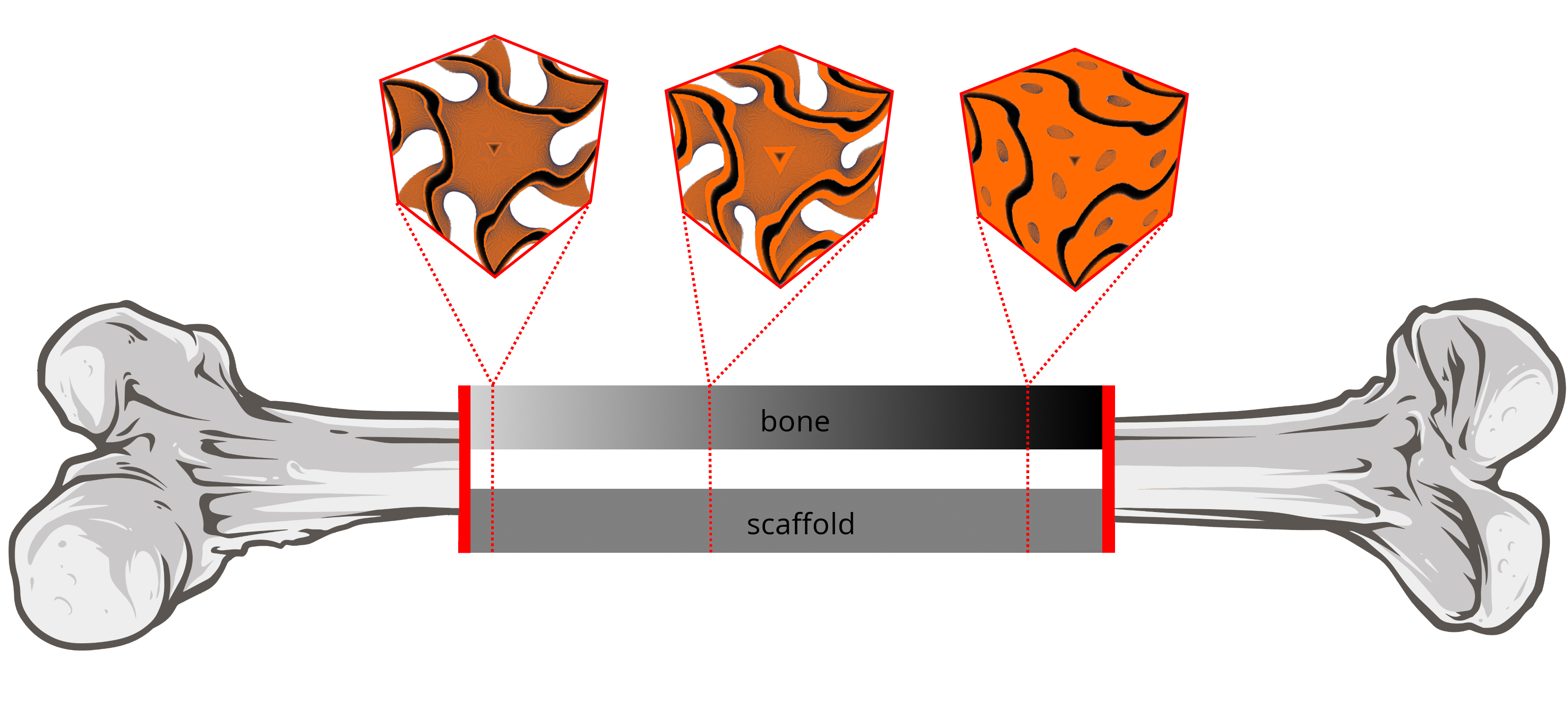}
    \caption{Visualization of the micro-macro-coupling approach where we have a spatially varying bone density (color gradient within defect site) and a constant scaffold density. For each combination of scaffold and bone density one obtains a specific microstructure for which one has to solve the microscopic problem. Above the defect site three different Gyroid microstructures are depicted with gradually increasing bone volume fraction but constant scaffold volume fraction.}
    \label{fig:micro_macro_coupling}
\end{figure}

Roughly, the regeneration process is described in a phenomenological approach by modeling mechanically stimulated cell interactions. To do so, we model in total four different cell types, \emph{progenitor} or \emph{stem cells} (pro), \emph{fibroblasts} (fib), \emph{chondrocytes} (cho) and \emph{osteoblasts} (ost). In \Cref{fig:cell_interactions} the different cell interactions are depicted. Depending on the present stimulus all cells can at least die (decrease in cell population) or proliferate (increase in cell population; modeled through logistic growth). Furthermore, the stem cells can differentiate into other cell types (decrease in stem cell population; increase in other cell population) and these are the only cells to do so. The fibroblasts and stem cells are additionally able to move in space (depicted by a double circle), which is modeled by a diffusion process.

For the mechanical aspect we employ linear elasticity, where the material properties are determined by the bone-scaffold composite. For simplicity, only the osteoblasts are responsible for bone formation. We note that our model is easily extensible to a more general system of cell interactions with bio-active molecules.

\begin{figure}[ht]
    \centering
    \begin{tikzpicture}[->, node distance=8em, every state/.style={draw, minimum size=1.2cm}]
        \node[state, double] (A) {pro};
        \node[state, double] (B) [right of=A] {fib};
        \node[state] (C) [below of=A] {cho};
        \node[state] (D) [below of=B] {ost};
    
        \path (A) edge node[above] {diff.} (B)
              (A) edge node[left] {diff.} (C)
              (A) edge node[right] {diff.} (D);
    
        \path (A) edge[loop above] node[above] {apo. $-$} (A)
              (B) edge[loop above] node[above] {apo. $-$} (B)
              (C) edge[loop below] node[below] {apo. $-$} (C)
              (D) edge[loop below] node[below] {apo. $-$} (D);
    
         \path (A) edge[loop left] node[left] {prolif. $+$} (A)
              (B) edge[loop right] node[right] {prolif. $+$} (B)
              (C) edge[loop left] node[left] {prolif. $+$} (C)
              (D) edge[loop right] node[right] {prolif. $+$} (D);
    
        \node[draw, rectangle, rounded corners, fit=(A) (B) (C) (D), inner xsep=6.3em, inner ysep=5.1em, label=above:{Bio-mechanic Cells}] (Cells) {}; %
    
        \node[draw, rectangle, rounded corners, inner sep=1em] (LinElast) [right=3em of Cells] {Linear Elasticity};
    
        \path (LinElast) edge (Cells);
    \end{tikzpicture}
    \caption{Interaction of cells. The mechanical stimulation given through linear elasticity, depending on this stimulus the cells can either \emph{differentiate}, \emph{proliferate} or die (\emph{apoptosis}). Cells with a double circle are also able to migrate within the bone defect.}
    \label{fig:cell_interactions}
\end{figure}
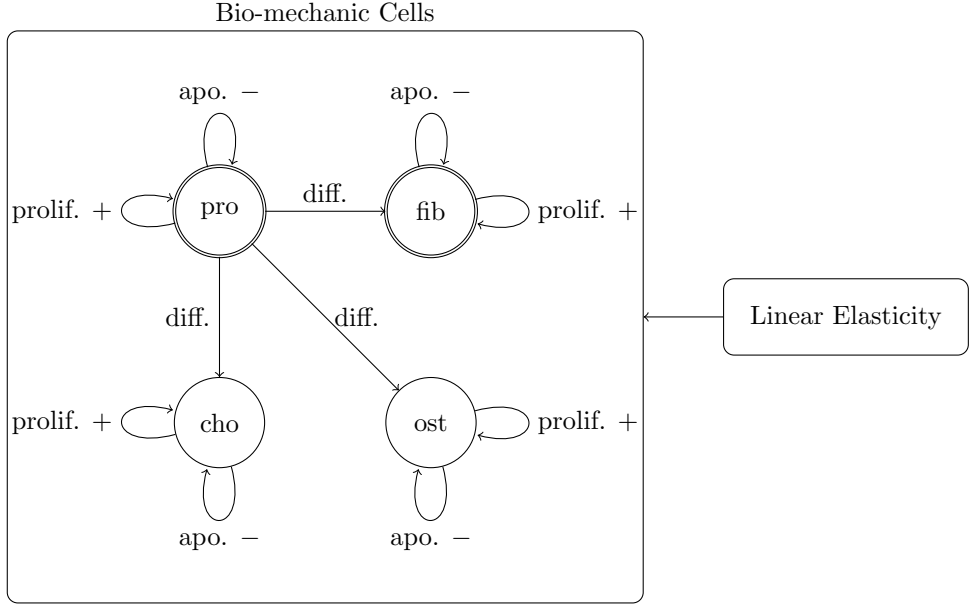

\subsection*{The Macroscopic System of Equations}

The computational domain $\Omega \subset \R^3$ refers not only to the bone defect, which is denoted by $\Omega_\square \subset \Omega$, but also includes a fixator stabilizing the bone defect. Therefore, on both sides of the defect, healthy bone is assumed to be present. Healthy bone is described as a cylindrical ring (i.e., \emph{cortical bone}), denoted by $\Omega_\text{cortical} \subset \Omega$, and inside this ring bone marrow ($\Omega_\text{marrow} \subset \Omega$) is located. Note, that the computational domain is only extended such that we can account for the mechanics regarding the fixator more accurately; cell signaling and behavior is only simulated within the defect. The outer surface of the cortical bone is called the \emph{periosteum} and referred to by $\Gamma_P$. By $I = [0, T]$ we denote a finite regeneration time interval. Furthermore, we introduce the function $\rho(x)$ describing scaffold volume in the defect side for $x \in \Omega_\square$. We decouple spatial and time dependence of the scaffold and only account for bulk erosion by a time-dependent exponential decay factor $\sigma(t) = \exp(-k_1 t)$, such that the product $\sigma(t) \rho(x)$ describes the material properties of the scaffold in time and space.

The material properties of the bone-scaffold composite are described by the stiffness tensor $\stiffness(\sigma(t) \cdot \rho, \cost)$. In the following we will omit the factor $\sigma(t)$ and only write $\rho(x)$. For the macroscopic model we only assume isotropy and ellipticity of the tensor. By $u(t, x)$ we describe the displacement satisfying the mechanical equilibrium below and by $\sym u$ we denote the symmetric gradient of the displacement, i.e., the strain.

Thus, the coupled system of linear elasticity (describing the mechanical environment), reaction-diffusion equations and logistic growth ODEs for the bio-active molecules reads as follows.

\begin{align*}
    \Style{DDisplayFunc=outset,DShorten=false}
    0 &= -\diver(\stiffness(\sigma(t) \cdot \rho, \cost) \sym u)  \tag{mechanical equilibrium} \\
    \pderiv{c_i}{t} &= \diver(\diff_i(\sigma(t) \cdot \rho, \cost) \nabla c_i) + \cR_i(c_1, \ldots, c_N, u, \rho) \tag{reaction-diffusion equations}
    \intertext{where the reaction term $\cR_i$ has the form}
    \cR_i &= \underbrace{k^\mathrm{prolif}_i(\oct) c_i \left(1 - \frac{\sum_i c_i}{1 - \rho}\right)}_{\text{proliferation}} + \underbrace{\mu_i(\oct) \kdiffpro(\oct) \cpro}_{\text{differentiation}} - \underbrace{k^\mathrm{apo}_i(\oct) c_i}_{\text{apoptosis}}.
\end{align*}

For the diffusion coefficients we set $\diff = \diff_\mathrm{pro} = \diff_\mathrm{fi} = \migspeed (1 - \rho)$ and $\diff_\mathrm{cho} = \diff_\mathrm{ost} \equiv 0$. Considering the differentiation of the cells, we set $\mu_\mathrm{pro} \equiv -1$ and assume $\mufi + \mucho + \muost \equiv 1$. The remaining functional relationships are made explicit in \Cref{tab:functional_rel}.

Finally, we also need to introduce boundary conditions. To this end, we assume progenitor cells to be located on the periosteum and inside the bone marrow. From this interface between healthy bone and the defect, progenitor cells are able to migrate into the defect site. For linear elasticity the Dirichlet boundary $\Gamma_D$ is the surface of the distal bone side, which is fully constrained and let the Neumann boundary $\Gamma_N$ be the proximal bone side, where an axial compressive load and two tangential forces inducing bending are imposed. In summary, our boundary conditions read:
\begin{align*}
    \left(\stiffness(\rho(x), \cost(x)) : \sym u\right) \cdot \nu &= g_N(x), && \text{on } \Gamma_N, \\
    u(t, x) &= g_D(x), && \text{on } \Gamma_D,\\
    c_i(0, x) &= 0, && \text{in } \Omega_\square,\\
    \cpro(t, x) &= 0.3, && \text{in } \Omega_\text{marrow} \text{ and on } \Gamma_P, \\
    \cost(t, x) &= 1, && \text{in } \Omega_\text{cortical}, \\
    D \nabla c_i(t, x) \cdot \nu &= 0, && \text{in } \Omega_\square \text{ not adjacent to bone},
\end{align*}
where $\nu$ refers to the outer unit normal of the respective surfaces.

\subsection*{The Microscopic System of Equations}

For the microscopic problem we assume the microstructure to admit a periodic pattern for which we can then apply \emph{periodic homogenization}. Roughly, this procedure seeks to find effective coefficients approximating the heterogeneous structure by a homogeneous one. Mathematically this is done by letting the period of the periodic medium approach $0$. This reduces the problem size drastically since we then only have constant material properties instead of spatially varying ones. In the following presentation we largely follow the results presented by \cite{Allaire}. But since our mechanical stimulus depends non-linearly on the strain, special care has to be taken here to find the actual homogenized stimulus. By $x$ we will denote points in the macroscopic space and by $y$ points in the microscopic space.

In summary, we need to consider microscopic models for the diffusion and linear elasticity. Heuristically, we need to employ periodic homogenization everywhere, where we have spatial derivatives of the microstructural quantities, i.e.\ $\rho$ and $\cost$.

We will start by presenting the microscopic model for linear elasticity. To this end let $Y = [0, 1]^3$ be the rescaled \emph{unit periodicity cell}. By $\cM(Y)$ we denote the admissible \emph{Hooke's laws} consisting of fourth order tensors $\stiffnessmicro(y)$. In fact, we need to solve a microscopic problem for each macroscopic space $x$. Here the material properties are mainly determined by the density values of the scaffolding structure $\rho$ and the osteoblasts $\cost$, in \Cref{eq:micro_elastic_modulus} we will show how to determine Hooke's law rigorously for our model. Altogether, our Hooke's law actually has the form $\stiffnessmicro(x, y) \in L^2(\Omega, \cM(Y))$.

Thus, for each macroscopic point $x \in \Omega$ we need to solve the following \emph{cell problem}
\begin{empheq}[left=\empheqlbrace]{equation}
    \begin{aligned}
        -\diver_y(\stiffnessmicro(x, y)(e_{ij} + \sym_y(w_{ij}(y)))) = 0 \quad &\text{in Y} \\
        y \mapsto w_{ij}(y) \quad & Y\text{periodic},
    \end{aligned}
    \label{eq:linelast_cell_problem}
\end{empheq}
where $e_{ij} = e_i \otimes_\mathrm{s} e_j$ denotes a basis for the symmetric matrices and $w_{ij}$ are called the corresponding \emph{correctors}.

From this, one can then derive the effective elastic coefficient as
\begin{equation}
    \stiffnesshom_{ijkl}(x) = \Int{\stiffnessmicro(x, y)(e_{ij} + \sym_y(w_{ij})) : (e_{kl} + \sym_y(w_{kl}))}{y, Y}.
    \label{eq:hom_elastic_modulus}
\end{equation}

Within the homogenization for linear elasticity there also arises the problem of finding the effective stimulus $S^\mathrm{hom} = \lim_{\varepsilon \to 0} S(\micstrain)$, since it depends on the strain $\sym u_\varepsilon$, where $\varepsilon$ denotes the period of the periodic structure. Unfortunately, due to the non-linearity of the stimulus functions, this does not simply converge to $S(\macstrain) \neq S^\mathrm{hom} $ when $\varepsilon$ approaches $0$. This is due to the the microscopic displacement $u_\varepsilon$ converging only weakly to the homogenized displacement $\uhom$. The problem becomes more apparent when we are considering the actual functional relationships used for the cellular activities, there we do not only consider the octahedral shear strain but prepend a strongly non-linear function $f$ to it. We are therefore considering functions of the type $f \circ \oct$. To tackle this, one needs to include the correctors ($w_{ij}$ in \Cref{eq:linelast_cell_problem}),
into the calculation of the homogenized stimulus $S^\mathrm{hom}$, such that we can obtain a strong convergence result for the microscopic and homogenized strains. More precisely, we make use of the following result for weak limits of rapidly oscillating periodic functions.

\begin{lemma}[Weak limits for rapidly oscillating periodic functions, see \cite{WeakOsc,IntroHomog}]
    Let $f \in L^p(\Omega; L^\infty_\#(Y))$ for some $p \in [1, \infty)$. Then $f_\varepsilon(x) \coloneqq f(x, x/\varepsilon) \in L^p(\Omega)$ with 
    \begin{align*}
        \norm{f_\varepsilon}_{L^p(\Omega)} \leq \norm{f}_{L^p(\Omega; L^\infty_\#(Y))}, \; \text{and} \\
        f_\varepsilon \weakc \Int{f(x, y)}{y,Y} \text{ in $L^p(\Omega)$}.
    \end{align*}
    \label{lem:weak_limits}
\end{lemma}
Furthermore, we make use of the following corrector result within homogenization theory to obtain strong convergence. To this end, we introduce the \emph{corrector tensor} $G_\varepsilon$ defined by
\begin{empheq}[left=\empheqlbrace]{equation}
    \begin{aligned}
        G_\varepsilon(x) &= G\left( \frac{x}{\varepsilon} \right), && \text{a.e. on } \Omega \\
        G_{ijkl}(y) &= \sym_{ij} \chi_{kl}(y), && \text{a.e. on } Y,
    \end{aligned}
    \label{eq:corrector_tensor}
\end{empheq}
where $\chi_{kl}$ is given by $\chi_{kl} \coloneqq w_{kl} + P_{kl}$, with $P_{kl}$ such that $\sym P_{kl} = e_{kl}$. Take note, that in \cite{IntroHomog} they define $\chi_{kl} \coloneqq -w_{kl} + P_{kl}$, this discrepancy is due to their correctors corresponding to the negative unit strains, i.e.\ $-e_{kl}$.
\begin{theorem}[Corrector Result, \cite{IntroHomog}]
    Let $u_\varepsilon$ be the solution to the microscopic problem. Then
    \begin{equation}
        \sym u_\varepsilon - G_\varepsilon \macstrain \to 0 \; \text{strongly in } (L^1(\Omega))^{3 \times 3}.
    \end{equation}
    \label{thm:corrector_result}
\end{theorem}
Combining \Cref{lem:weak_limits,thm:corrector_result} we obtain the following result for the homogenized stimulus.

\begin{theorem}[Homogenized Stimulus]
    Let  $\Omega$ be bounded and $S \colon \R^{3\times3} \to \R$ be Lipschitz and bounded. Then it holds
    \begin{equation}
        S(\sym u_\varepsilon) \stackrel{\varepsilon \to 0}{\weakc} \Int{S(G(y) \macstrain)}{y,Y} \eqqcolon S^{\mathrm{hom}} \text{ in } L^2(\Omega).
    \end{equation}
    \label{thm:hom_stimulus}
\end{theorem}

\begin{proof}
    The functions that we seek to consider have the structure
    \begin{align*}
        T \colon &\Omega \to L^\infty_\#(Y) \\
        & x \mapsto \left( y \mapsto S \left( \sym u_\varepsilon(y; x) \right) \right).
    \end{align*}
    Due to the Lipschitz continuity and \Cref{thm:corrector_result} we obtain
    \begin{align*}
        &\abs{S(\sym u_\varepsilon(x/\varepsilon; x)) - S(G(x/\varepsilon; x) \macstrain(x))} \\
        \leq L  &\norm{\sym u_\varepsilon(x/\varepsilon; x) - G(x/\varepsilon; x) \macstrain(x)} \longrightarrow 0 \quad \text{strongly in } L^1(\Omega)^{3 \times 3}.
    \end{align*}
    Therefore, we can equivalently also consider the \emph{corrected} function
    \begin{align*}
        T \colon &\Omega \to L^\infty_\#(Y) \\
        & x \mapsto \left( y \mapsto S \left( G(y; x) \macstrain(x) \right) \right).
    \end{align*}
    We have $\macstrain \in L^2(\Omega)$, $\sym w_{kl} \in L^2_\#(Y)$ and thus for fixed $x \in \Omega$ it holds that $G(y; x) \in L^2_\#(Y)$. Due to the boundedness of $S$ and $Y$, we directly get $T(x) \in L^\infty_\#(Y)$.
    The boundedness of $\Omega$ then yields $T \in L^2(\Omega, L^\infty_\#(\Omega))$ since 
    \[
        \Int{\norm{T(x)}_{L^\infty_\#(Y)}^2}{x, \Omega} \leq C^2 \abs{\Omega}.%
    \]
    Combining everything and applying \Cref{lem:weak_limits} we arrive at
    \[
       S(\sym u_\varepsilon(x/\varepsilon; x)) = S(G(x/\varepsilon; x) \macstrain(x)) \weakc \Int{S(G(y; x) \macstrain(x))}{y, Y},
    \]
    concluding the proof.
\end{proof}

\begin{remark}
    In our setting, the stimulus functions have the form $S = f \circ \oct$, where $f$ is either a mollifier or a sigmoidal function and thus bounded. Furthermore, $f$ and $\oct$ are both Lipschitz and therefore fulfill the requirements to apply \Cref{thm:hom_stimulus}. Computationally, this is very expensive though, because for each functional relationship the homogenized version of it has to be computed.
\end{remark}

As a final step in our homogenization procedure, we need to consider the diffusion process, where we again follow the presentation of \cite{Allaire}. This reduces to the standard conductivity problem
\[
    -\diver(\diff(x) \nabla w) = 0
\]
and for each unit conductivity $e_i$, we need to solve (again for each macroscopic point) the following cell problem
\begin{equation}
    -\diver(\diffmic(x, y) (e_i + \nabla w)) = 0.
\end{equation}
We are then able to compute the effective diffusion coefficient $\diffhom$ as
\begin{equation}
    \diffhom_{ij}(x) = \Int{\diffmic(x, y) (e_i + \nabla_y w_i) \cdot (e_j + \nabla_y w_j)}{y, Y}.
    \label{eq:hom_diffusivity}
\end{equation}

In the next section, \Cref{sec:micro_macro_coupling}, we will make the choices for the elastic moduli and diffusivity more concrete.

In total, in the weak sense, the (periodically) homogenized model reads as follows.

\begin{empheq}[left=\empheqlbrace]{align*}
    \Style{DDisplayFunc=outset,DShorten=false}
    0 &= -\diver(\stiffnesshom(\sigma(t) \cdot \rho, \cost^\mathrm{hom}) \macstrain), \\
    \pderiv{c^\mathrm{hom}_i}{t} &= \diver(\diffhom_i(\sigma(t) \cdot \rho, \cost^\mathrm{hom}) \nabla c^\mathrm{hom}_i) + \cR^\mathrm{hom}_i(c^\mathrm{hom}_1, \ldots, c^\mathrm{hom}_N, u^\mathrm{hom}, \rho),\\
    \cR^\mathrm{hom}_i &= \underbrace{(k^\mathrm{prolif}_i(\oct))^\mathrm{hom} c^\mathrm{hom}_i \left(1 - \frac{\sum_i c^\mathrm{hom}_i}{1 - \rho}\right)}_{\text{proliferation}} \\ &+ \underbrace{(\mu_i(\oct) \kdiffpro(\oct))^\mathrm{hom} \cpro^\mathrm{hom}}_{\text{differentiation}} \\ &- \underbrace{(k^\mathrm{apo}_i(\oct))^\mathrm{hom} c^\mathrm{hom}_i}_{\text{apoptosis}}.
\end{empheq}

\section{Micro-Macro Coupling}\label{sec:micro_macro_coupling}

This section is devoted to rigorously presenting the choice for the elastic moduli and diffusivity of the microscopic model. Throughout this work, we assume that osteoblasts only grow on the surface of the scaffolding structure (see \Cref{fig:bone_scaffold_composite} for an illustration of the microscopic bone-scaffold composite). To this end, we make an explicit choice to the geometry of the underlying microstructure, in general there are multiple different approaches that can be considered. In this work though, we will restrict our attention to two main geometries, the \emph{strut-like} and the \emph{Gyroid} structure. The former geometry arises due to its canonical and simple nature, whereas the second one seeks to mimic the bone microarchitecture as its found in real bone of animals. In \Cref{fig:base_structures}, these two types of geometry are depicted for one unit cell.

\begin{figure}[ht]
    \centering
    \begin{subfigure}[c]{0.4\textwidth}
        \includegraphics[width=\textwidth]{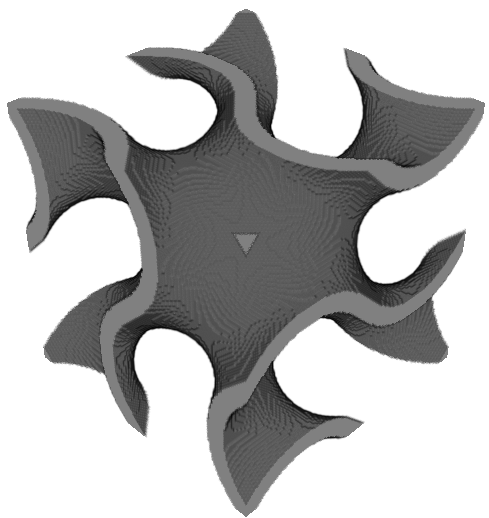}
        \caption{Gyroid}
    \end{subfigure}
     \begin{subfigure}[c]{0.4\textwidth}
        \includegraphics[width=\textwidth]{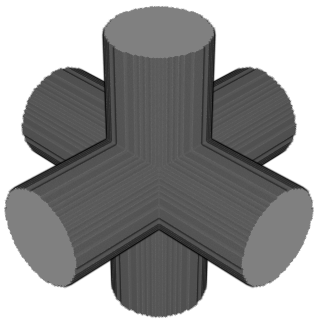}
        \caption{Strut-like}
    \end{subfigure}
    \caption{The two base geometries that will be used throughout this work.}
    \label{fig:base_structures}
\end{figure}

The Gyroid structure \cite{Gyroid,TPMSTissue} exhibits several advantageous properties beyond its biomimetic geometry: it is a \emph{triply periodic minimal surface} (TPMS), possesses high stiffness, and features continuous curvature, thus preventing stress concentrations and establishing a stable mechanical environment. Furthermore, it has high permeability and due to its open architecture facilitates cell infiltration into the scaffold.

\begin{figure}[ht]
    \centering
    \begin{subfigure}[c]{0.4\textwidth}
        \includegraphics[width=\textwidth]{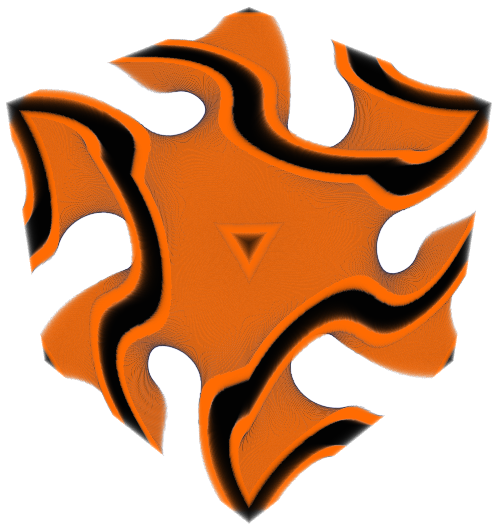}
        \caption{Gyroid}
    \end{subfigure}
     \begin{subfigure}[c]{0.4\textwidth}
        \includegraphics[width=\textwidth]{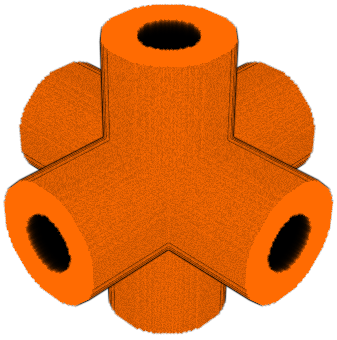}
        \caption{Strut-like}
    \end{subfigure}
    \caption{Illustration of how the bone (orange) grows only on the surface of the scaffolding structure (black) and thus forming the bone-scaffold composite on the microscopical level.}
    \label{fig:bone_scaffold_composite}
\end{figure}

In fact, the elastic moduli and diffusivity only depend on the values of the scaffold $\rho$ and the osteoblasts $\cost$. To this end, we define our geometry as an implicit surface with a certain thickness.
Let $f$ be the implicit function determining the surface, in the case of the Gyroid and Strut-like structure we have

\begin{minipage}{0.4\textwidth}
    \centering
   \begin{align*}
        f(x,y,z) \coloneqq &\cos(2\pi x) \sin(2\pi y) \\
        +& \cos(2\pi y) \sin(2\pi z) \\
        +& \cos(2\pi z) \sin(2\pi x).
    \end{align*}
    Gyroid.
\end{minipage}%
\begin{minipage}{0.55\textwidth}
    \centering
   \begin{align*}
        f(x,y,z) \coloneqq & \min((x - 0.5)^2 + (y - 0.5)^2, \\
        &(x - 0.5)^2 + (z - 0.5)^2,\\
        &(y - 0.5)^2 + (z - 0.5^2)).
    \end{align*}
    Strut-like.
\end{minipage}

\vspace*{1em}
Note that the provided implicit function for the Gyroid is an approximation; however, it represents a highly accurate one.

The implicit surface is then given as
\[
    \Gamma \coloneqq \setc*{(x, y, z) \in Y}{f(x, y, z) = 0}.
\]
To give the surface a certain thickness and to make the bone only grow on top of the surface of the scaffolding structure, we introduce the thickness parameters $\alpha, \beta \geq 0$ and define the scaffold and bone, respectively, as
\begin{align*}
    \cS_\alpha &\coloneqq \setc*{y \in Y}{\abs{f(y)} \leq \alpha},\\
    \cB_{\alpha,\beta} &\coloneqq \setc*{y \in Y}{\alpha \leq \abs{f(y)} \leq \beta}.
\end{align*}
Since $\rho$ and $\cost$ corresponds to volume densities, we need to find a map that maps volume densities to the thickness parameters. That is, for $\alpha$ we need to find the inverse function $\varrho^{-1}$ to 
\begin{equation*}
    \varrho \colon \alpha \mapsto \Int{\mathds{1}_{\cS_\alpha}}{y,Y}.
\end{equation*}
For determining the parameter $\beta$ we proceed similarly and consider the map
\[
    \varrho \colon (\alpha, \beta) \mapsto \Int{\mathds{1}_{\cB_{\alpha, \beta}}}{y,Y}.
\]
To summarize, we proceed along the following scheme to determine the correct values of $\alpha$ and $\beta$ corresponding to their volume densities for the scaffold and bone geometry:
\begin{enumerate}
    \item Retrieve the value for $\alpha$ by means of the mapping $\varrho^{-1}(\rho)$.
    \item Then evaluate the function $\varrho^{-1}(\alpha, \cost)$ to get $\beta$, where $\alpha$ denotes the value calculated in the first step.
\end{enumerate}
In total we can define the volumes to thicknesses mapping as
\[
    \mathrm{geometry} \colon (\rho, \cost) \mapsto (\varrho^{-1}(\rho), \varrho^{-1}(\varrho^{-1}(\rho), \cost)).
\]
Take note that this offers a general scheme for general structures, such as the strut-like geometry or for example also the \emph{Schwarz P} (which is a smooth approximation to the strut-like structure) and \emph{Diamond} structures, see e.g.~\cite{GyroidDondl}. In the following we will omit the mappings to retrieve the values and just write $\alpha$ and $\beta$ corresponding to the volume densities.

For the macroscopic quantities $\rho$ and $\cost$ we are now able to define the microscopic elastic modulus as 
\begin{equation}
    \begin{aligned}
        \stiffnessmicro(\rho, \cost, y) &= \mathds{1}_{\cS_\alpha} \cC(\lambda_\rho, \mu_\rho) + \mathds{1}_{\cB_{\alpha, \beta}} \cC(\lambda_{\mathrm{ost}}, \mu_{\mathrm{ost}}) \\
    \cC(\lambda, \mu) \varepsilon : \varepsilon &= 2 \mu \varepsilon : \varepsilon + \lambda \Tr(\varepsilon)\Tr(\varepsilon),
    \end{aligned}
    \label{eq:micro_elastic_modulus}
\end{equation}
where $\lambda$ and $\mu$ correspond to the Lam\'{e} parameters of the respective material.

For the diffusivity we can proceed similarly, we first need to determine the diffusivity corresponding to the migration speed of the cells denoted by $\migspeed$. %
Then we obtain the microscopic diffusion coefficient as
\begin{equation}
    \diffmic(\rho, \cost, y) = \migspeed (\mathds{1}_{\cS_\alpha} + \mathds{1}_{\cB_{\alpha,\beta}}) = \migspeed \mathds{1}_{\cS_\beta},
    \label{eq:micro_diffusivity}
\end{equation}
that is, the material matrix is given as the combination of the bone and scaffold geometry, such that the cells are not able to migrate there.
Note, that this is a simplified assumption on how bone grows and might differ from reality in the sense that the bone does not grow uniformly on the surface of the scaffolding structure; instead one would also need to consider the curvature of the surface which affects bone growth \cite{BoneCurvature}.

Now we have all the ingredients to perform the micro-macro coupling by using periodic homogenization. To do so, one needs to proceed along the following scheme (see also \Cref{fig:micro_macro_coupling}):
\begin{enumerate}
    \item Let $x \in \Omega$ be in macroscopic space. Determine the macroscopic density values $\rho(x)$ and $\cost(x)$ for scaffold and bone, respectively.
    \item Map the densities to the thickness parameters $\alpha$ and $\beta$ and determine the microscopic elastic modulus $\stiffnessmicro$ and diffusivity $\diffmic$ (see \Cref{eq:micro_elastic_modulus,eq:micro_diffusivity}, respectively).
    \item Perform periodic homogenization to obtain the effective coefficients $\stiffnesshom(x)$ and $\diffhom(x)$ for each integration point $x$, for linear elasticity one also needs the correctors (see \Cref{eq:hom_elastic_modulus,eq:hom_diffusivity}, respectively).
    \item Plug in the homogenized coefficients into the macroscopic system of equations and first solve only for the homogenized displacement.
    \item Calculate the effective functional relationships $(f \circ \oct)^\mathrm{hom}$ (see \Cref{thm:hom_stimulus}) using the previously calculated correctors.
    \item Solve for homogenized cell fields under consideration of the effective stimulus.
\end{enumerate}
In literature this procedure is referred to by FE$^2$, because you apply the FEM in the microscopic space to get the effective behavior, and with these information you solve on the macroscopic scale, also by the FEM, to get the overall response. In general, you can choose different approaches per computational scale. For example, in this work we employ FE-FFT for linear elasticity and FE$^2$ for diffusion.

\section{Implementation Details}
\label{sec:implementation_details}

In the following we will make the choice for the parameters and solving details in our simulation concrete. For solving the PDE system we use \lstinline{firedrake} \cite{FiredrakeUserManual} which offers a general purpose framework for solving PDEs by means of the FEM. There, we make extensive use of \lstinline{ExternalOperator}s \cite{ExternalOp1,ExternalOp2} which offer the ability to include user-defined expressions that are not part of the unified form language \lstinline{ufl} \cite{ufl}. The homogenized coefficients for linear elasticity and diffusion are precalculated for discrete samples of the densities of the scaffolding and bone fields. These results are then saved onto the disk and belong to the offline part of our simulation. For the online calculations these calculated coefficients are then read from disk and a nearest neighbor interpolation on the densities is done to extract the appropriate homogenized quantities and respective correctors to perform online calculation of the effective stimulus. This is done by means of \lstinline{ExternalOperator}s in each macroscopic space point.

As indicated earlier, we will use FFT-based methods for solving the microscopic linear elasticity problem. To this end, the cell problem is recast into an integral equation, the so called \emph{periodic Lippmann-Schwinger} equation. This can then be solved by means of the basic FFT-based fixed-point scheme \cite{Moulinec}. But since we are dealing with high contrastive material this scheme is insufficient, and we use more sophisticated solving approaches, that is, we employ the conjugate gradient (CG) method and the discretization is done on a staggered grid \cite{CompHom,FFTReview}.

For solving the macroscopic linear elastic system efficiently, we employed a Krylov subspace method (GMRES) combined with a multigrid preconditioner based on the geometric algebraic multigrid (GAMG) approach.
The coupled system of cell interactions is more challenging to solve, since the (anisotropic and heterogeneous) diffusion introduces numerical oscillations which leads to physically implausible solutions, i.e.\ the cell densities overshoot and undershoot the range $[0, 1]$. To tackle this, the problem is recast into variational inequalities for which we then use a semi-smooth Newton method with cell densities in the appropriate range $[0, 1]$ as constraints. This approach largely follows the work presented in \cite{VariationalInequalities}. Using a semi-smooth Newton method enables us to calculate the necessary adjoint derivatives using generalized differentials and therefore does not introduce additional complexity when solving the adjoint equation for the optimization problem; see \Cref{sec:optimization}.

For the time-discretization, we employ a semi-implicit Euler-scheme due to its numerical stability and accuracy, allowing larger time-step sizes without compromising stability.

Our computational domain is taken from \cite{MahdiJaber} where an external fixator is explicitly modeled for providing a controlled mechanical environment that supports healing. Furthermore, this allows comparability of the approaches presented by \cite{MahdiJaber} and the one presented in this work. In short, we want to simulate the scaffold mediated bone growth for a rat femur with a diameter of \SI{2}{\milli\m} with a critical bone defect replicated by a \SI{5}{\milli\m}-wide gap in which the scaffolding structure is located. For the material properties also coincide with those mentioned in \cite{MahdiJaber}, i.e.\ for the fixator we use Polyether-ether-ketone: PEEK ($E = \SI{3800}{\mega\pascal}$, $\nu = 0.3$), for the nails titanium: Ti ($E = \SI{111000}{\mega\pascal}$, $\nu = 0.33$) and for the scaffold Polycaprolactone: PCL ($E = \SI{350}{\mega\pascal}$, $\nu = 0.33$). In contrast to their work, we do not differentiate between immature and mature bone, but only consider mature bone ($E = \SI{5000}{\mega\pascal}$, $\nu=0.3$).

For the mechanical loading conditions (\enquote{aimed to simulate peak load under normal walking conditions}) the distal part was fully constrained (homogeneous Dirichlet boundary condition) and for the proximal bone side an axial compressive load of \SI{14.7}{\newton} and two tangential forces of \SI{1.8}{\newton} (inducing bending loads) were applied (inhomogeneous traction boundary condition).
For the Neumann boundary $\Gamma_N$ we only consider the part belonging to the cortical bone and do not apply mechanical loading on the bone marrow.
This leads to the following boundary conditions for the mechanical equilibrium
\begin{align*}
    u &\equiv 0 \quad \text{on } \Gamma_D \\
    \sigma \cdot \nu &= (\sigma_\mathrm{comp} + \sigma_\mathrm{bend}) \cdot \nu \quad \text{on } \Gamma_N,
\end{align*}
where
\begin{equation*}
    \sigma_\mathrm{comp} = \begin{pmatrix}
-\frac{\SI{14.7}{\newton}}{\abs{\Gamma_N}} & 0 & 0\\
0 & 0 & 0 \\
0 & 0 & 0
\end{pmatrix}, \quad \sigma_\mathrm{bend} = \begin{pmatrix}
0 & -\frac{\SI{1.8}{\newton}}{\abs{\Gamma_N}} & \frac{\SI{1.8}{\newton}}{\abs{\Gamma_N}} \\
0 & 0 & 0 \\
0 & 0 & 0
\end{pmatrix}.
\end{equation*}

Now we turn our attention to determining the functional relationships, see \Cref{tab:functional_rel}, based on the cell rates given in \cite[Table 3]{MahdiJaber}. Following \cite{Checa}, mechanical stimulation of a specific cell type resulted in maximal proliferation of that cell type, no proliferation and maximal apoptosis of the other cell types and also maximal differentiation of stem cells into the specific cell phenotype.
Let $S = \frac{\oct}{a}$, $a \coloneqq 0.0375$ \cite{Huiskes}, denote our mechanical stimulus, then for the progenitor cell differentiation we obtain the following stimulus-dependent rules (see \Cref{table:mechano_algorithm}).
\begin{table}[ht]
    \centering
    \begin{tabular}{llll} 
     \textbf{Bone resorption} & \textbf{Osteoblasts} & \textbf{Chondrocytes} & \textbf{Fibroblasts} \\[1ex]
     $S \leq 0.01$ & $0.01 < S \leq 3$ & $3 < S \leq 5$ & $5 < S$ \\
    \end{tabular}
    \caption{Mechano-regulation algorithms for progenitor cell differentiation, adapted from \cite[Table 2]{MahdiJaber}.}
    \label{table:mechano_algorithm}
\end{table}
For the forward model testing the stimulus in this discontinuous manner is perfectly fine, but since we want to perform optimization of the scaffolding structure differentiable testing functions are necessary. To this end, we approximate the indicator functions $\mathds{1}_{[a, b]}$ and $\mathds{1}_{[a, \infty)}$ by means of mollifiers and sigmoidal functions, respectively.

\begin{table}[ht]
    \centering
    \begin{tabular}{lccc}\toprule
        $k_{\mathrm{cell}}^{\mathrm{type}}(S)$ & \multicolumn{3}{c}{Type}\\ \cmidrule{2-4}
        Cell & prolif & diff & apo \\ \midrule
        pro & $0.6 \cdot \sigma_{0.01}$ & $-\ln(0.7) \cdot \sigma_{0.01}$ & $-\ln(0.95) \cdot (1 - \sigma_{0.01})$ \\
        fib & $0.55 \cdot \sigma_5$ & --- & $-\ln(0.95) \cdot (1 - \sigma_5)$ \\
        cho & $0.2 \cdot \phi_{3, 5}$ & --- & $-\ln(0.9) \cdot (1 - \phi_{3, 5})$ \\
        ost & $0.3 \cdot \phi_{0.01, 3}$ & --- & $-\ln(0.84) \cdot (1 - \phi_{0.01, 3})$ \\ \bottomrule
    \end{tabular}
    \caption{Functional relationships for cell interactions.}\label{tab:functional_rel}
\end{table}

For the diffusion coefficient $\migspeed$ we chose the value $\num{6e-4}$, such that the defect side is fully infiltrated by progenitor cells within $\SI{16}{\week}$. To obtain similar osteoblast evolution as in \cite{MahdiJaber}, one could set $\migspeed$ to \num{1.2e-3}, but this effectively only compresses the simulation time into \SI{84}{\day}.

\section{Optimization}
\label{sec:optimization}

As mentioned in the previous section, we use the framework \lstinline{firedrake} to solve our system. But the choice for using this framework actually lies in its ability to perform automatic differentiation and thus enables us to perform PDE constrained optimization relatively straightforward. One only needs to provide the (adjoint) derivatives for the user-defined expressions (\lstinline{ExternalOperator}s). To be more precise, we want to use PDE constrained optimization (see \cite{PDEOpt} for a thorough introduction to this topic) using the adjoint approach, so that we find a scaffolding structure which stimulates bone regeneration the most.

For the ad-hoc homogenized model a rigorous mathematical justification for the optimization has been provided in \cite{ScaffoldOpt}. For this work we can largely follow the approach and results presented there. In contrast to their work, we additionally perform periodic homogenization and need to provide the appropriate derivatives of the effective quantities. To this end, we make extensive use of the following \emph{variational characterization} of the effective elastic modulus and diffusion coefficient, which significantly simplifies the respective derivatives. More precisely, following \cite[Eq. (1.14) and Eq (1.23)]{Allaire}, we have 
\begin{align*}
    \diffhom \xi \cdot \xi &= \min_{w(y) \in H^1_\#(Y)} \Int{\diffmic(y) (\xi + \nabla_y w) \cdot (\xi + \nabla_y w)}{y, Y}, \tag{diffusion} \\
    \stiffnesshom \xi : \xi &= \min_{w(y) \in H^1_\#(Y)^3} \Int{\stiffnessmicro(y) (\xi + \sym_y w) : (\xi + \sym_y w)}{y, Y} \tag{elasticity}.
\end{align*}
This means that the correctors are given as the minimizers of a functional and due to this property the derivative of the functional evaluated at the corrector is zero. From this, and since we calculated the correctors for all unit strains, we obtain that the homogenized quantities have derivatives $\GD_\rho$ and $\GD_{\cost}$ equal to $0$.

Now it only remains to calculate the action of the adjoint derivative for the homogenized stimulus. To this end, let
\begin{equation*}
    N \colon (\strain, \rho, \mathrm{b}) \mapsto \Int{f \circ \oct (G_{\rho, \mathrm{b}} : \strain)}{y, Y}
\end{equation*}
be the operator that maps macroscopic strain, scaffold and bone density to the corresponding homogenized stimulus, where $G_{\rho, \mathrm{b}}$ denotes the corrector tensor from \Cref{eq:corrector_tensor}. Here, $f$ denotes a cell-specific activation function, which we made concrete in \Cref{tab:functional_rel}.

We start by determining the derivative to the octahedral shear strain 
\begin{align}
    \oct(\strain) &= \frac{2}{3 a} \sqrt{3 \Tr \strain^2 - \Tr^2 \strain}\notag
\intertext{which is given as}
    \GD \oct(\strain_0) \delta\strain &= \frac4{3 a^2} \frac1{\oct(\strain_0)} \left( \Tr \strain_0 \delta\strain - \frac13 \Tr \strain_0 \Tr \delta\strain \right).
\end{align}

Thus, we can determine the derivative of $N$ with respect to the strain $\strain$ as
\begin{equation}
    \GD_\strain N(\strain_0) \delta\strain = \Int{f^\prime(\oct(G(y) : \strain_0)) \GD \oct(G(y) : \strain_0)(G(y) : \delta\strain)}{y,Y}.
\end{equation}
Calculating the derivatives with respect to the scaffold and bone densities is a little more challenging since we need to differentiate the microstructure with respect to the densities. In fact, these derivatives could be calculated by means of distributional derivatives and methods from calculus of moving surfaces, but this is computationally not tractable.
Here, we instead focus on approximating the microstructure derivatives by central differences.
That is, we have
\begin{equation}
    \GD_\rho N(\rho_0) \delta\rho = \Int{f^\prime(\oct(G(y) : \strain_0)) \GD \oct(G(y) : \strain_0) (\GD_\rho G_{\rho_0, \mathrm{b}_0} (y) \delta\rho : \strain_0)}{y, Y}
\end{equation}
and substitute the expression $\GD_\rho G_{\rho_0, \mathrm{b}_0} (y) \delta\rho$ with central differences as
\begin{equation*}
    \GD_\rho G_{\rho_0, \mathrm{b}_0} (y) \delta\rho \approx \frac{G_{\scaffold_0 + \frac{h}2, \bone_0} - G_{\scaffold_0 - \frac{h}2, \bone_0}}{h}.
\end{equation*}
For the bone derivative we proceed analogously.

Their adjoint actions are given as
\begin{align*}
    \GD^\ast_\strain N(\strain_0) \colon \R^\ast &\to \left(\R^{3 \times 3} \right)^\ast, \\
    \delta y^\ast &\mapsto \delta y^\ast \circ \GD_\strain N(\strain_0) \\ &= \left( \delta\strain \mapsto \delta y^\ast(\GD_\strain N(\strain_0) \delta\strain) \right) \\
    &= (\delta\strain \mapsto \delta y^\ast(1) \GD_\strain N(\strain_0) \delta\strain ), \\
    \GD^\ast_\scaffold N(\scaffold_0) \colon \R^\ast &\to \R^\ast, \\
    \delta y^\ast & \mapsto (\delta\scaffold \mapsto \delta y^\ast(1) \GD_\scaffold N(\scaffold_0) \delta\scaffold), \\
    \GD^\ast_\bone N(\bone_0) \colon \R^\ast &\to \R^\ast, \\
    \delta y^\ast & \mapsto (\delta\bone \mapsto \delta y^\ast(1) \GD_\bone N(\scaffold_0) \delta\bone).
\end{align*}

For the class of objective functionals that can be considered within this framework, we refer to \cite{ScaffoldOpt}. However, in this work, we mainly focus on a mixture of the objectives to minimize \emph{compliance} and to maximize \emph{amount of regenerated bone at end time}. Considering only the amount of bone regenerated, the scaffolds seek a low density, making them difficult to manufacture. Therefore, the set of controls is restricted by a box constraint of $[c_P, C_P]$, for two fixed constants $c_P, C_P > 0$. For our simulations, these constants are set to $c_P = 0.1$ and $C_P = 0.99$.
Another approach to prevent scaffolds from becoming too porous is to also consider compliance in the objective such that the resulting structures admit appropriate stiffness. The final objective functional then reads as
\begin{equation}
    J(y, \rho) \coloneqq \gamma \max_{t \in [0, T]} \left( w(t) \underbrace{\Int{\sigma(t, x) : \sym u(t, x)}{x, \Omega}}_{\text{compliance}} \right) - \eta \Int{\cost(T, x)}{x, \Omega}
    \label{eq:objective_functional}
\end{equation}
The factors $\gamma \geq 0$ and $\eta \geq 0$ denote the weighing of the respective objectives and $w(t)$ is to select specific time points. Here, we choose $w(\{0, T\}) = 1$ and $w(t) = 0$ otherwise. As mentioned in \cite{ScaffoldOpt}, the maximum is smoothly approximated by $L^p(I)$ norms for large values of $p$. 

\section{Simulations}
\label{sec:simulations}

In what follows, we will distinguish between the following three models \textsc{n} (none), \textsc{ed} (elasticity \& diffusion) and \textsc{eds} (elasticity, diffusion \& stimulus) which refer to the systems where no underlying homogenization is employed, where only the diffusion and linear elastic equations are homogenized, and the fully homogenized model, respectively. For each model, we want to examine how the cell densities evolve for a simulation time of \SI{140}{\day} and a time step of $\Delta t = \SI{1}{\day}$. The simulation time of \SI{140}{\day} is chosen so that the amount of regenerated bone slows down and stabilizes and a convergent trend for the osteoblasts is observed. Note that for all simulations that follow, we chose an initial scaffold porosity of \SI{79}{\percent}. In \Cref{fig:cell_evolution_opt} we will compare the cell evolution of the initial homogeneous scaffolding structure with the optimized scaffolds; see \Cref{sec:opt_scaffolds}. Since the three models gradually try to describe bone growth more accurately, they also strongly differ in their running times. For the mode \textsc{n} the forward model needs about \SI{10}{\minute} and for each optimization step the same amount of time is needed. For all three modes $3$ to $4$ optimization steps were necessary, thus leading to a total execution time of $\sim~\SI{30}{\minute}$ until the optimal scaffolding structure is found. The forward models for the modes \textsc{ed} and \textsc{eds} needed about \SI{20}{\min} and \SI{3.5}{\hour} leading to total execution times of $\sim\SI{1.5}{\hour}$ and $\sim\SI{15}{\hour}$, respectively.

\subsection{Cell Evolution}

In \Cref{fig:cell_evolution} the different cell evolutions for the modes \textsc{n}, \textsc{ed} and \textsc{eds} are plotted. Here we consider the volume fraction of the respective cells within the defect side. For all cells except the progenitor cells we can see an increase in the cell population using the fully homogenized model \textsc{eds}. In particular, fibroblasts and chondrocytes are only stimulated when considering the microscopic stimulus. Since we have an increase in these cell populations we have a decrease in the progenitor cells' population. For the osteoblasts the evolution admits a similar trend but with higher rates of regenerated bone and the convergence behavior is attained at similar time.
\begin{figure}[ht]
\centering
\begin{subfigure}{.4\textwidth}
    \centering
    \includegraphics[width=\textwidth]{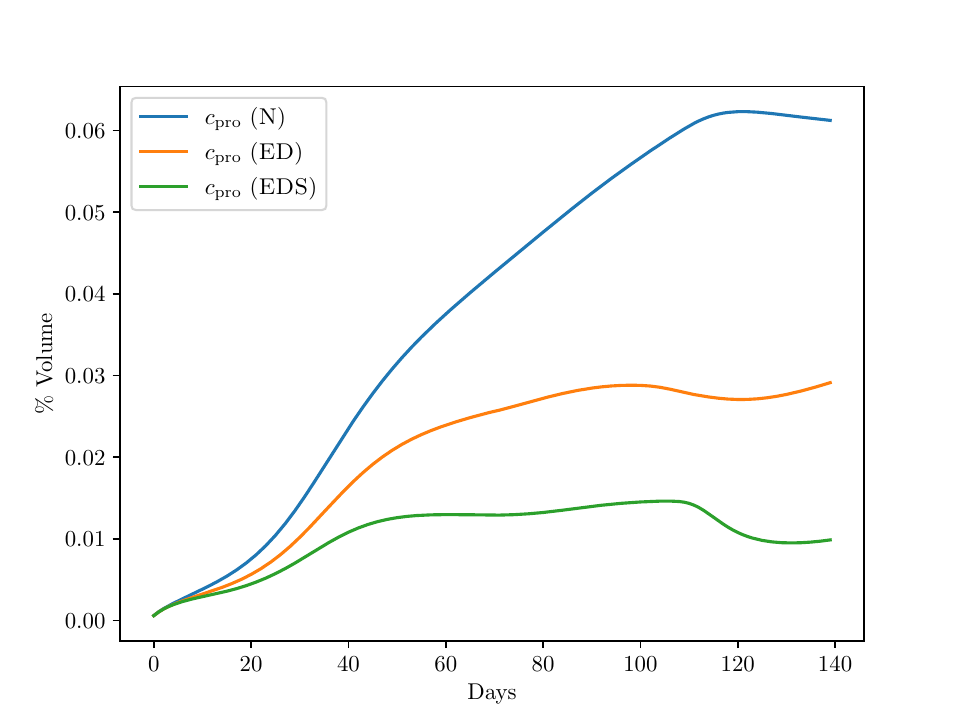}
    \caption{Progenitor cells}
\end{subfigure}%
\begin{subfigure}{.4\textwidth}
    \centering
    \includegraphics[width=\textwidth]{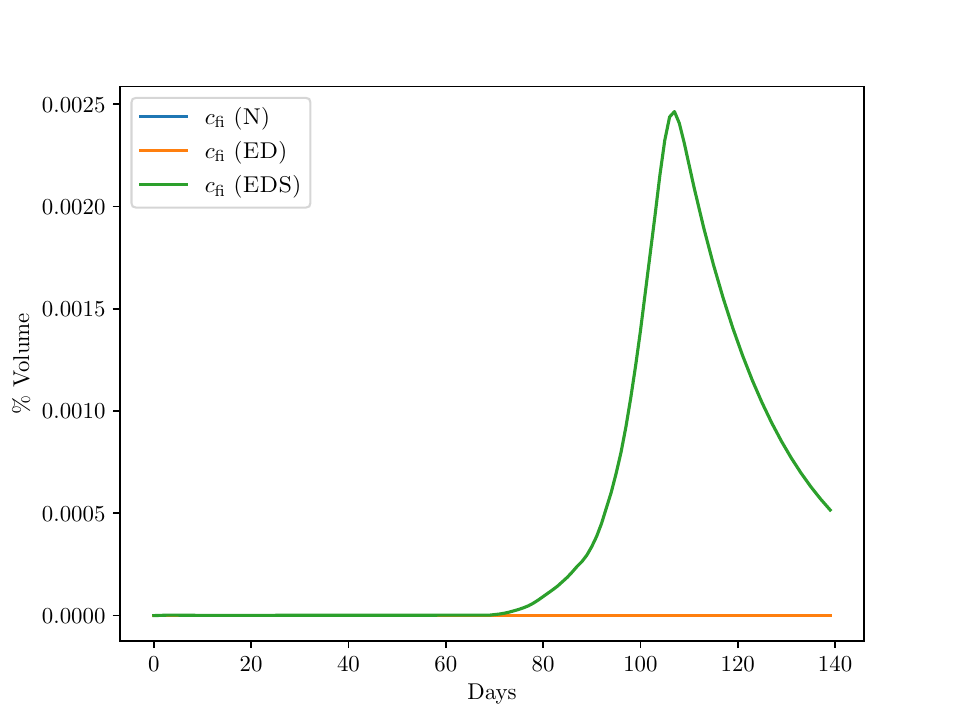}
    \caption{Fibroblasts}
\end{subfigure}
\begin{subfigure}{.4\textwidth}
    \centering
    \includegraphics[width=\textwidth]{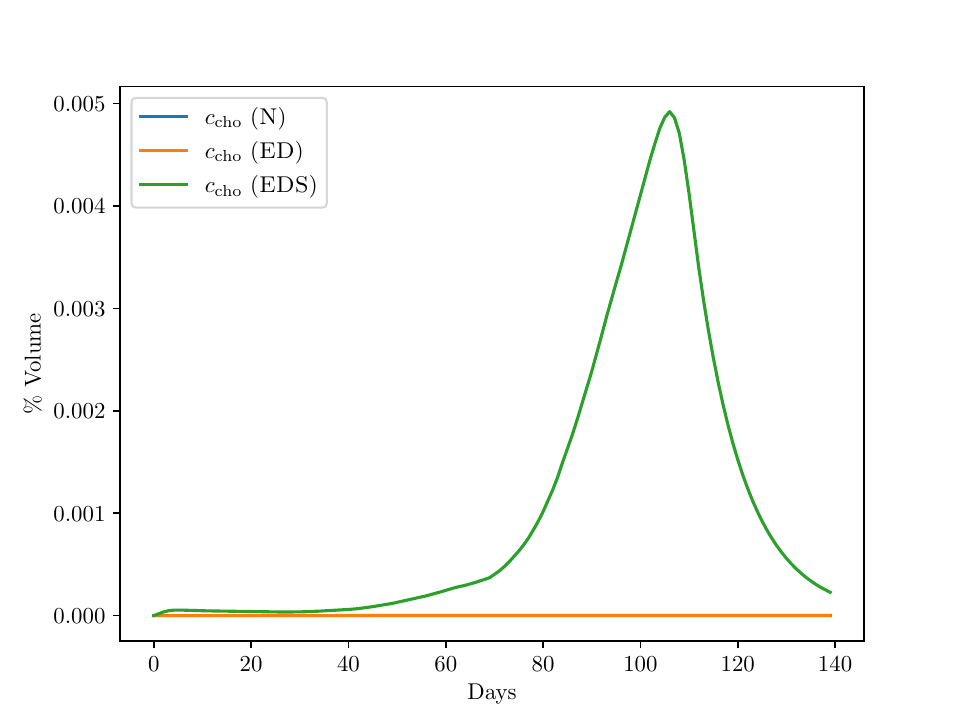}
    \caption{Chondrocytes}
\end{subfigure}%
\begin{subfigure}{.4\textwidth}
    \centering
    \includegraphics[width=\textwidth]{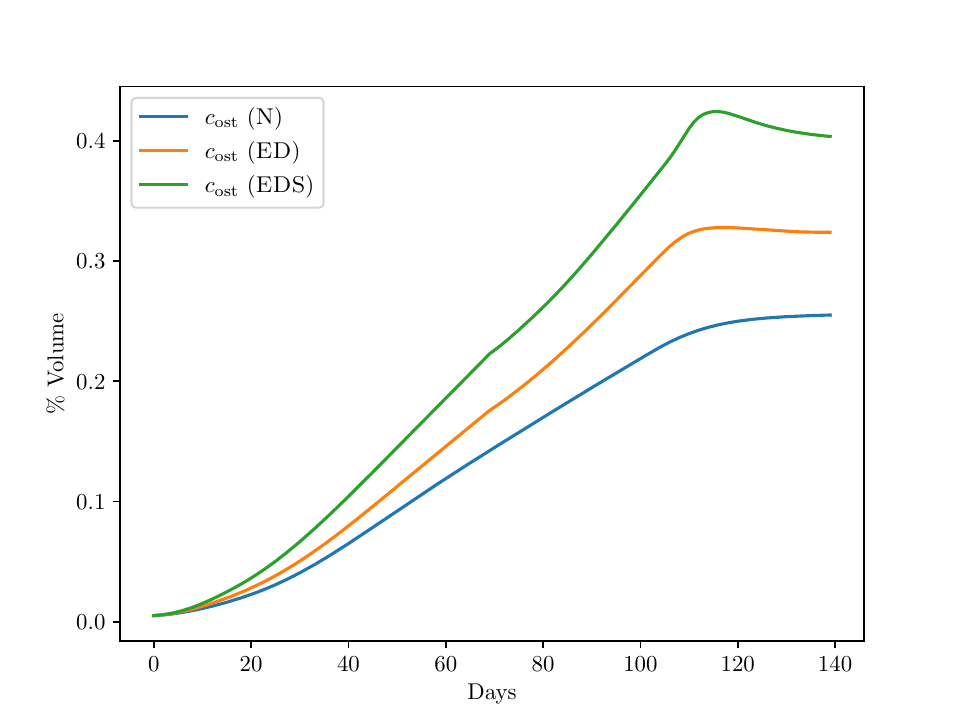}
    \caption{Osteoblasts}
\end{subfigure}
\caption{Comparison of the cell evolutions for the different modes \textsc{n}, \textsc{ed} and \textsc{eds} for a simulation time of \SI{140}{\day}.}
\label{fig:cell_evolution}
\end{figure}

\subsection{Optimized Scaffolding Structure}
\label{sec:opt_scaffolds}

We start by optimizing only the amount of bone regenerated at the end time, that is, we set $\gamma = 0$ and $\eta = 1$ in our objective functional \Cref{eq:objective_functional}. The results are depicted in \Cref{fig:optimized_scaffolds} for the modes \textsc{n}, \textsc{ed} and \textsc{eds}, where the fixator is located to the right hand side. All optimal scaffolds admit the same pattern to compensate for large stresses due to bending and compressive loads, where bone regeneration is much slower. At these points the scaffolding structure is denser to maintain an overall good mechanical performance. 

The scaffolding structures are reconstructed by converting the density volumes to the corresponding thickness values $\alpha$ (see \Cref{sec:micro_macro_coupling}) to obtain a spatially varying function $\alpha(x)$ throughout the defect volume. For this, we also need to specify the number of unit cells in the macroscopic space, where we chose $3$ unit cells per $\SI{4}{\mm}$.

\begin{figure}[ht]
    \centering
    \begin{subfigure}[t]{.3\textwidth}
        \centering
        \includegraphics[angle=90,width=0.7\textwidth]{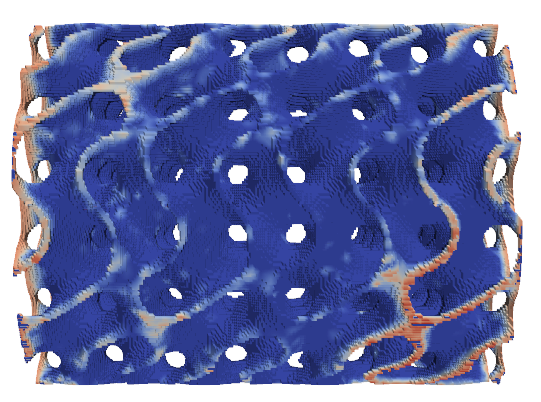}
        \caption{\textsc{n}}
    \end{subfigure}
    \hfill
    \begin{subfigure}[t]{.3\textwidth}
        \centering
        \includegraphics[angle=90,width=0.7\textwidth]{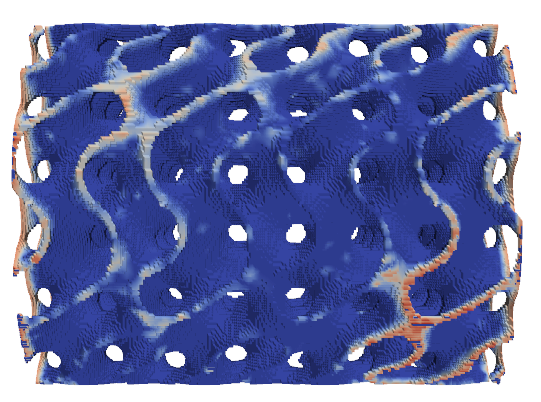}
        \caption{\textsc{ed}}
    \end{subfigure}
    \hfill
    \begin{subfigure}[t]{.3\textwidth}
        \centering
        \includegraphics[angle=90,width=0.7\textwidth]{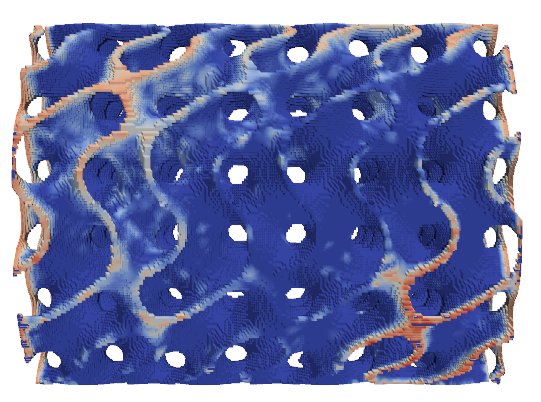}
        \caption{eds}
        \label{fig:optimized_scaffolds_eds}
    \end{subfigure}
    \caption{Optimized gyroid-based scaffolding structures for the different homogenized models \textsc{n}, \textsc{ed} and \textsc{eds} from left to right with the objective functional chosen to be the regenerated amount of bone at end time. The fixator is located to the right and the coloring denotes the density of the scaffolding structures.}
    \label{fig:optimized_scaffolds}
\end{figure}

Finally, we want to compare the cell evolutions for the initial homogeneous scaffold with a porosity of \SI{79}{\percent} with the optimized scaffold of \Cref{fig:optimized_scaffolds_eds}. Here we only employ the \textsc{eds} mode and the cell densities are depicted in \Cref{fig:cell_evolution_opt}. The maximal cell population for progenitor cells is almost doubled, leading to an increase in all cell densities. For the osteoblasts the trend is very similar until the point of convergence, where the regeneration process using the optimized scaffold surpasses that using a homogeneous scaffold.

\begin{figure}[ht]
\centering
\begin{subfigure}{.4\textwidth}
    \centering
    \includegraphics[width=\textwidth]{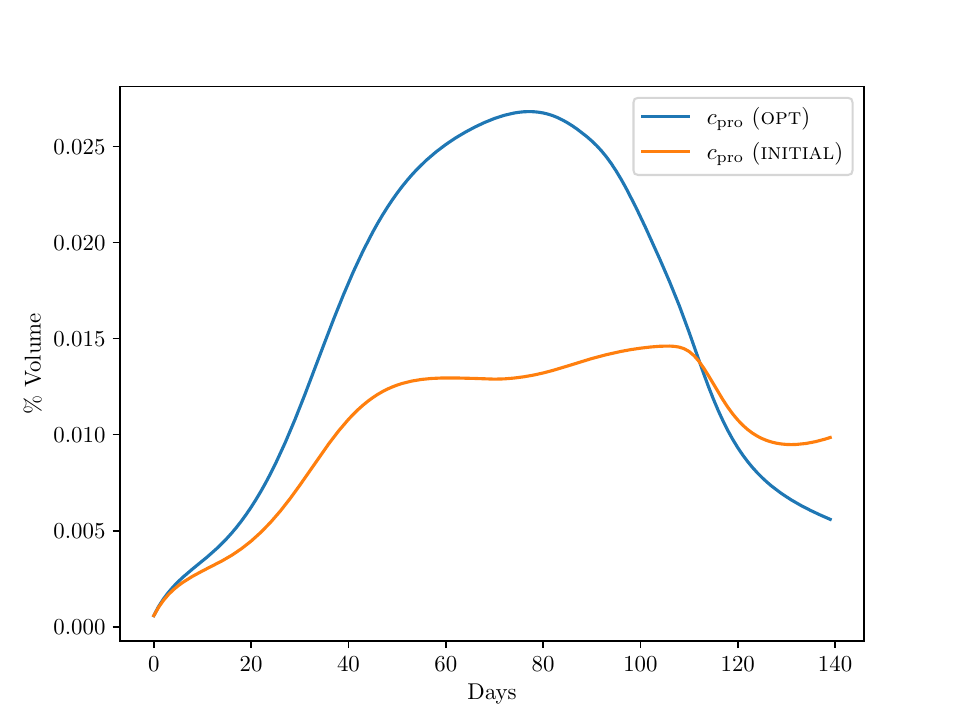}
    \caption{Progenitor cells}
\end{subfigure}%
\begin{subfigure}{.4\textwidth}
    \centering
    \includegraphics[width=\textwidth]{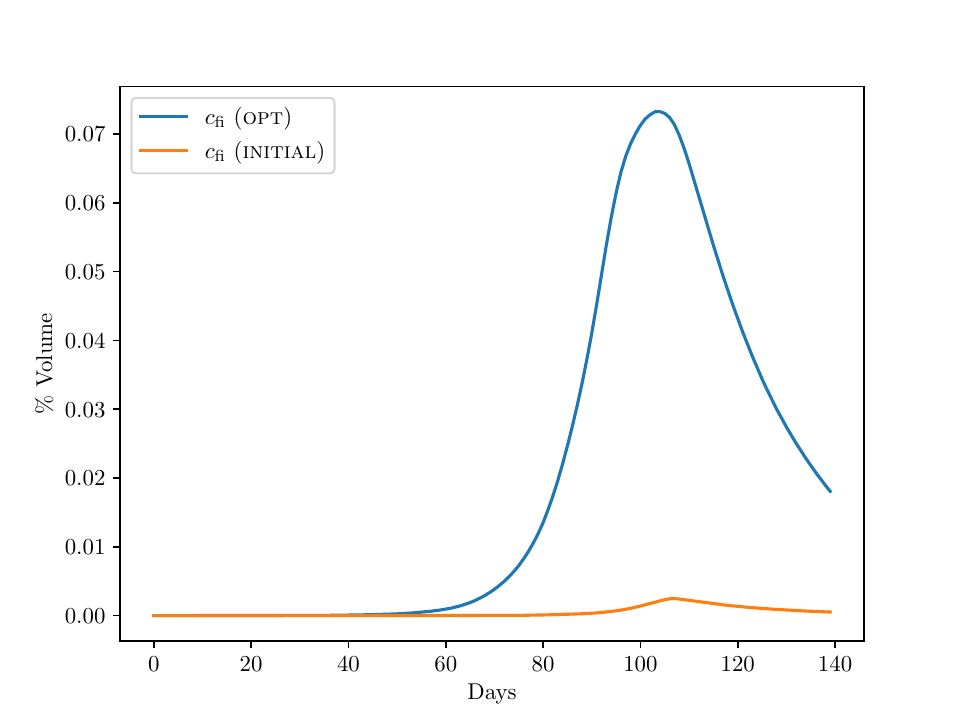}
    \caption{Fibroblasts}
\end{subfigure}
\begin{subfigure}{.4\textwidth}
    \centering
    \includegraphics[width=\textwidth]{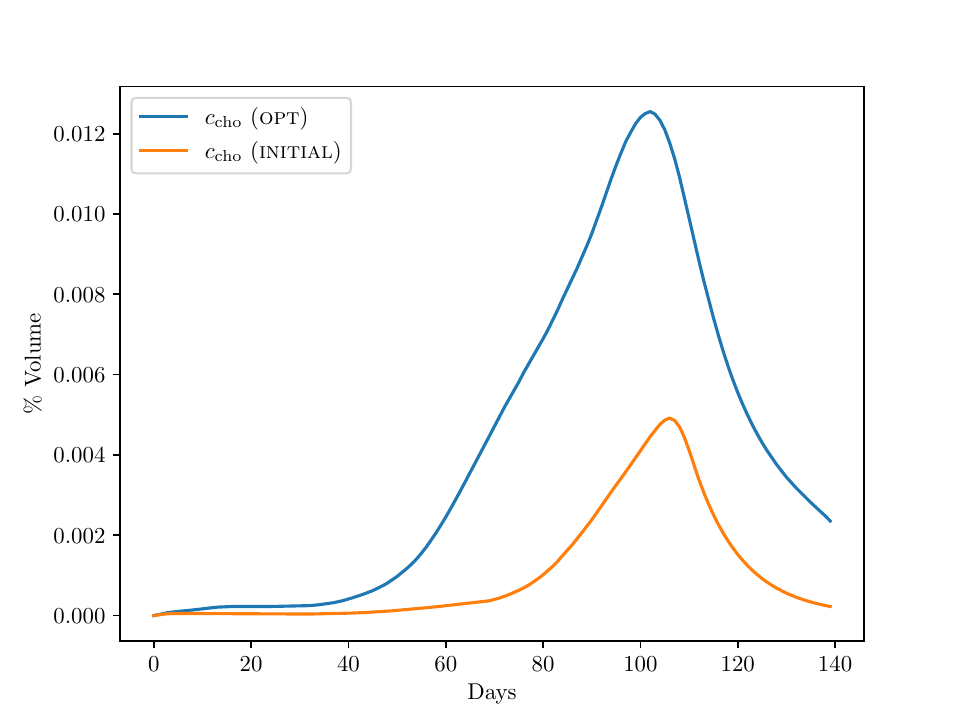}
    \caption{Chondrocytes}
\end{subfigure}%
\begin{subfigure}{.4\textwidth}
    \centering
    \includegraphics[width=\textwidth]{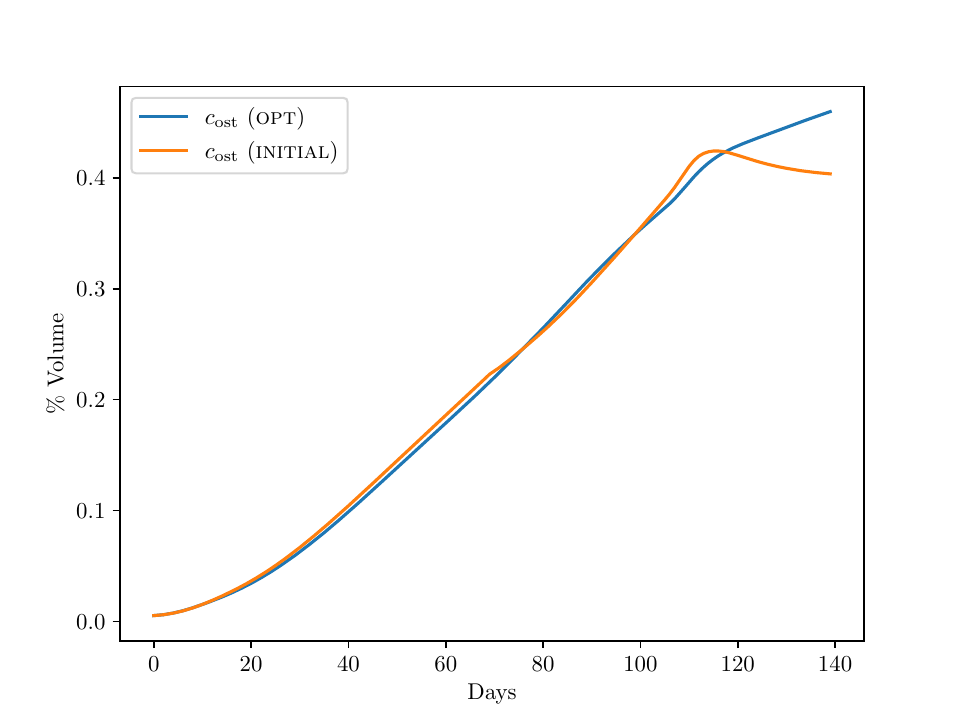}
    \caption{Osteoblasts}
\end{subfigure}
\caption{Comparison of the cell evolutions fo the \textsc{eds} mode of the homogeneous initial scaffold with a porosity of \SI{79}{\percent} (\textsc{initial}) with the optimized scaffold (\textsc{opt}) from above.}
\label{fig:cell_evolution_opt}
\end{figure}

In \Cref{fig:optimized_scaffolds_compliance} the optimized scaffolds for also including compliance (i.e.\ $\gamma, \eta =1$) in the objective functional are presented. Due to the box constraints in the optimization procedure, the difference between the optimized scaffolds with and without compliance minimization is marginal. One could also increase the weight for compliance minimization to make the differences more apparent.

\begin{figure}
    \centering
    \begin{subfigure}[t]{.3\textwidth}
        \centering
        \includegraphics[angle=90,width=0.7\textwidth]{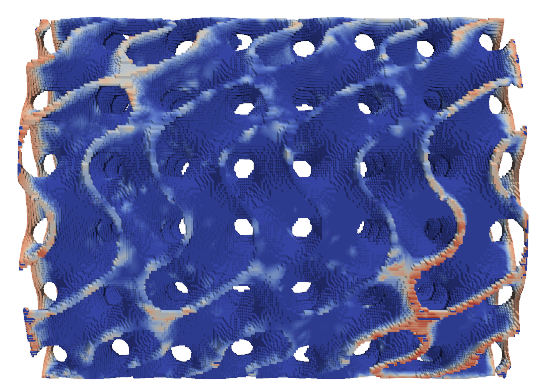}
        \caption{\textsc{n}}
    \end{subfigure}
    \hfill
    \begin{subfigure}[t]{.3\textwidth}
        \centering
        \includegraphics[angle=90,width=0.7\textwidth]{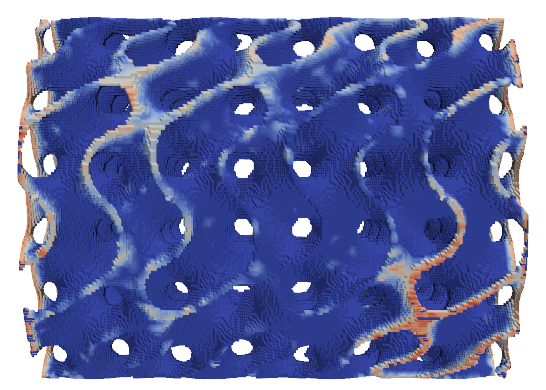}
        \caption{\textsc{ed}}
    \end{subfigure}
    \hfill
    \begin{subfigure}[t]{.3\textwidth}
        \centering
        \includegraphics[angle=90,width=0.7\textwidth]{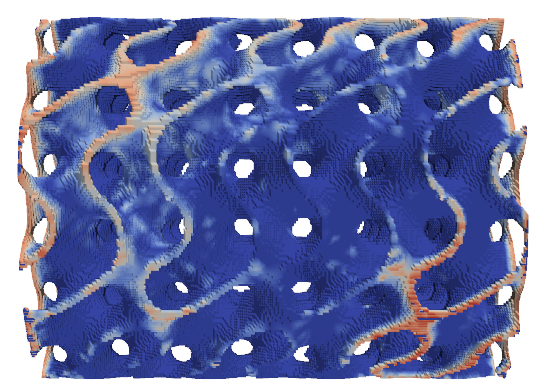}
        \caption{eds}
    \end{subfigure}
    \caption{Optimized gyroid-based scaffolding structures for the different homogenized models \textsc{n}, \textsc{ed} and \textsc{eds} from left to right with the objective functional chosen to be the regenerated amount of bone at end time and additionally including compliance to maintain mechanical integrity. The fixator is located to the right and the coloring denotes the density of the scaffolding structures.}
    \label{fig:optimized_scaffolds_compliance}
\end{figure}

In fact, if we insert the optimal scaffold calculated by means of the model \textsc{ed} into the forward model of \textsc{eds} the cell evolution behave the same as when using the optimal scaffold calculated by means of \textsc{eds}, indicating that this already gives a good candidate for the optimal structure with much shorter computation time. Therefore, we conclude that compensating for high stress areas is crucial for the optimization of the scaffold.

\Cref{fig:optimized_scaffolds_comparison} depicts two optimal scaffold designs for the Gyroid and strut-like based microgeometries computed within the fully homogenized model (\textsc{eds}). For both designs we see similar high stress areas compromising bone regeneration that need to be compensated by higher densities.
\begin{figure}[ht]
    \centering
    \begin{subfigure}[t]{.4\textwidth}
        \centering
        \includegraphics[angle=90,width=0.7\textwidth]{s_star_full_ss.png}
        \caption{Optimal Gyroid based scaffold}
    \end{subfigure}
    \hfill
    \begin{subfigure}[t]{.4\textwidth}
        \centering
        \includegraphics[angle=90,width=0.7\textwidth]{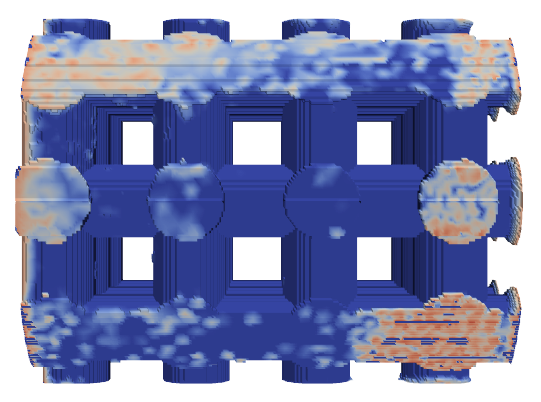}
        \caption{Optimal strut-like based scaffold}
    \end{subfigure}
    \caption{Optimized scaffolds for Gyroid and strut-like based microgeometries, calculated by means of the \textsc{eds} model.}
    \label{fig:optimized_scaffolds_comparison}
\end{figure}

We conclude this section by examining the differences in the cell evolutions for the Gyroid and strut-like micro geometries. These results are only compared for the models \textsc{ed} and \textsc{eds}, since for the model \textsc{n} the micro geometry has no impact on the regeneration process.
\begin{figure}[ht]
\centering
\begin{subfigure}{.4\textwidth}
    \centering
    \includegraphics[width=\textwidth]{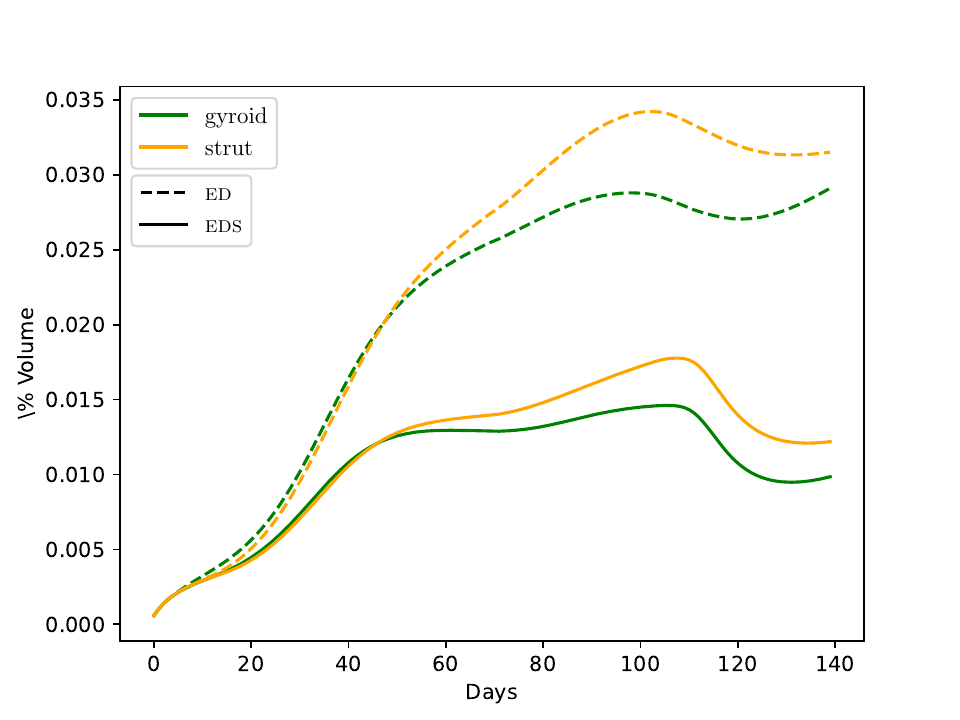}
    \caption{Progenitor cells}
\end{subfigure}%
\begin{subfigure}{.4\textwidth}
    \centering
    \includegraphics[width=\textwidth]{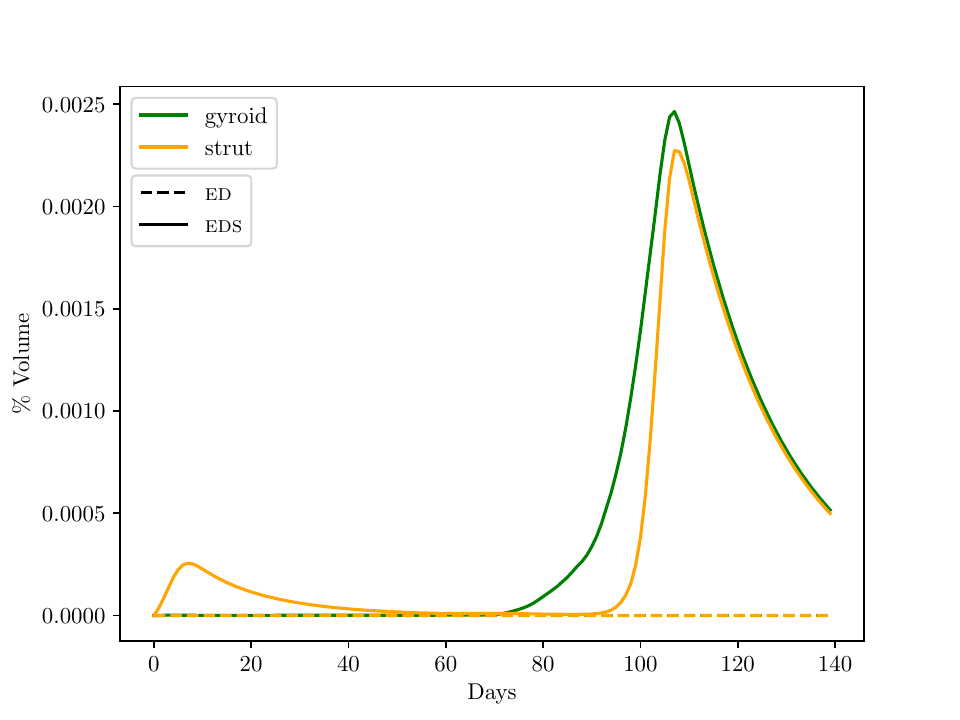}
    \caption{Fibroblasts}
\end{subfigure}
\begin{subfigure}{.4\textwidth}
    \centering
    \includegraphics[width=\textwidth]{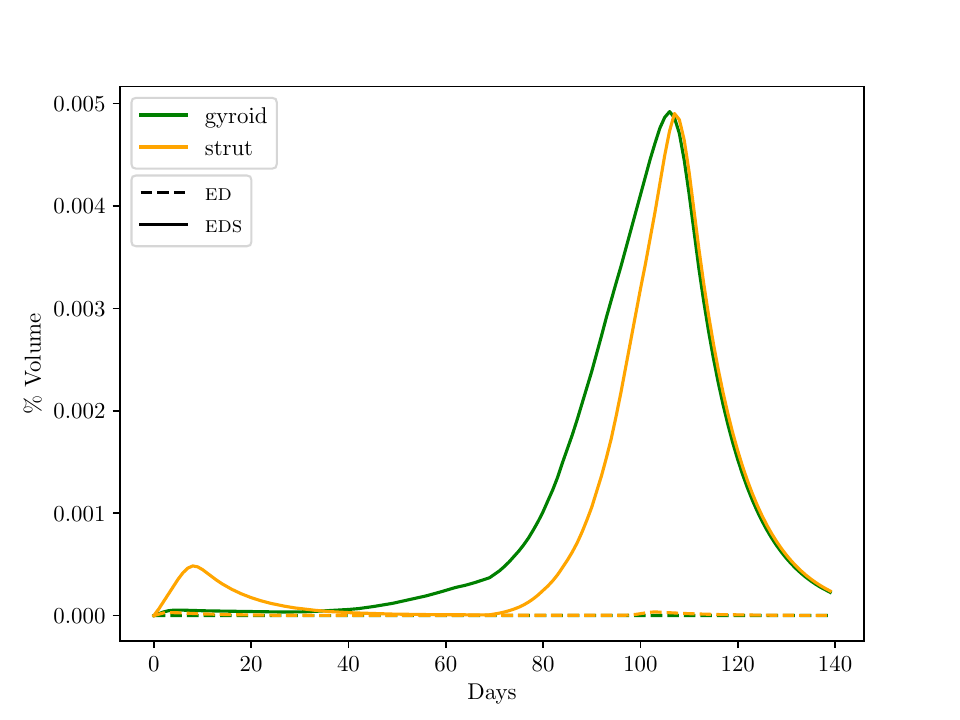}
    \caption{Chondrocytes}
\end{subfigure}%
\begin{subfigure}{.4\textwidth}
    \centering
    \includegraphics[width=\textwidth]{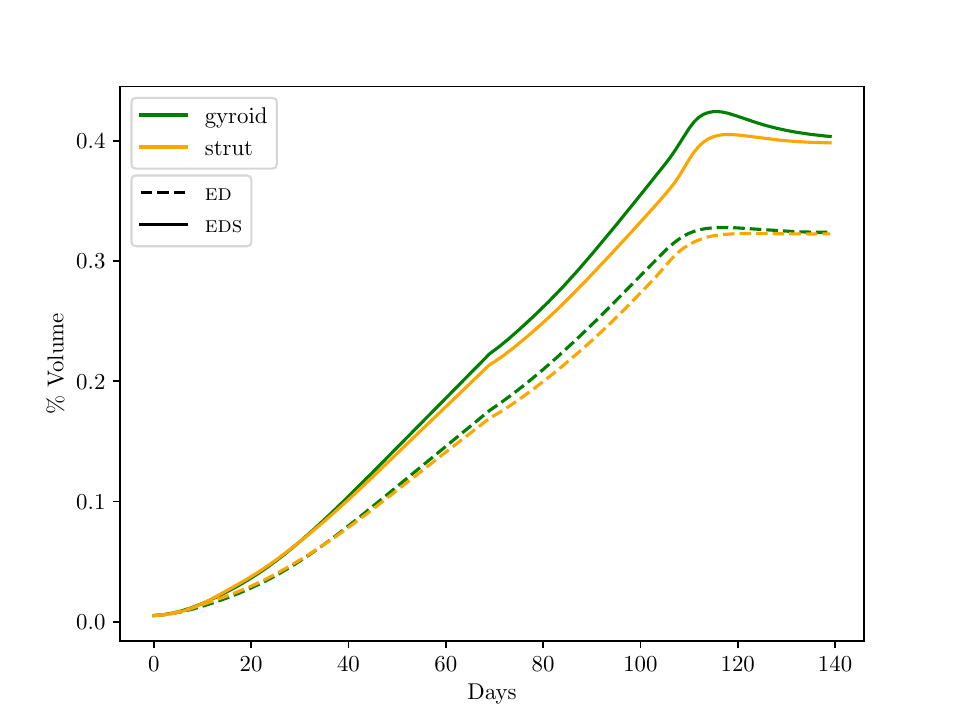}
    \caption{Osteoblasts}
\end{subfigure}
\caption{Comparison of the cell evolutions for the geometries Gyroid (green) and Strut (orange) for the modes \textsc{ed} (dashed) and \textsc{eds} (solid) for a simulation time of \SI{140}{\day}.}
\label{fig:cell_evolution_comparison}
\end{figure}

\section{Conclusion}

In this work we presented an extension to the model given in \cite{ScaffBone,ScaffoldOpt}. In contrast to the preceding work we explicitly considered the microstructural strain fluctations corresponding to a Gyroid and strut-like micro geometry by means of periodic homogenization. Even though we focused on the Gyroid and strut-like geometries the framework presented here permits easy integration of other geometries as well. Furthermore, we inferred that accounting for the microstructure is necessary since it strongly influences the regeneration process in the forward model. For the optimization of the scaffolds the stresses on the system had a stronger influence on the local volume fraction of the scaffold than the actual micro geometry. Therefore, the rather fast model \textsc{ED} gives a good basis for the ideal structure that stimulates bone growth the most, making it suitable for patient-specific therapy.

\section*{Acknowledgements}
The  authors  gratefully  acknowledge   support  from  BMBF within the e:Med program in the SyMBoD consortium (grant number 01ZX2210C).

\FloatBarrier

\printbibliography

\end{document}